\documentclass[a4paper,12pt,leqno]{amsart}
\usepackage{graphicx}
\usepackage[applemac]{inputenc}
\usepackage{selinput}
\usepackage{color}
\SelectInputMappings{
adieresis={ä},
ntilde={ñ}
}
\usepackage[spanish, english]{babel}
\usepackage{amssymb}
\usepackage{mathabx}
\usepackage{mathrsfs}
\usepackage{amsfonts,amscd,amsmath}
\usepackage{amssymb,latexsym}
\usepackage{enumitem} \numberwithin{equation}{section}

\def\irr#1{{\Irr}(#1)}
\def\irrp#1{{\Irr}_{p'}(#1)}
\def\irrq#1{{\Irr}_{q'}(#1)}

\def\irrpi#1{{\rm Irr}_{\pi'}(#1)}

\def\zent#1{{\bf Z}(#1)}

\def\nor{\triangleleft}
\def\det#1{{\rm det}(#1)}

\def\norm#1#2{{\bf N}_{#1}(#2)}

\def\aut#1{{\rm Aut}(#1)}

\def\out#1{{\rm Out}(#1)}

\let\phi=\varphi
\def\sbs{\subseteq}
\newcommand{\tw}[1]{{}^#1\!}
\newcommand{\NN}{{\mathbb{N}}}

\newcommand{\bG}{{\mathbf{G}}}

\newcommand{\cH}{{\mathcal H}}

\newcommand{\wt}[1]{\widetilde{#1}}
 \newcommand{\6}{^}

\newcommand{\Irr}{\operatorname{Irr}}
\newcommand{\GL}{{\operatorname{GL}}}
\newcommand{\GU}{{\operatorname{GU}}}
\newcommand{\SL}{{\operatorname{SL}}}

\newcommand{\PSL}{{\operatorname{PSL}}}
\newcommand{\SU}{{\operatorname{SU}}}

\let\la=\lambda

\newtheorem{lem}[subsection]{Lemma}
\newtheorem{cor}[subsection]{Corollary}
\newtheorem{thm}[subsection]{Theorem}
\newtheorem{prop}[subsection]{Proposition}

\newtheorem*{thm*}{Theorem}
\newtheorem*{thmA}{Theorem A}
\newtheorem*{thmB}{Theorem B}
\newtheorem*{thmC}{Theorem C}

\theoremstyle{definition}
\newtheorem{rem}[subsection]{Remark}
\theoremstyle{definition}

\theoremstyle{definition}

\theoremstyle{definition}

\theoremstyle{definition}

\def\N{\mathbb N}

\begin{document}

\keywords{Character degrees, perfect groups}

\subjclass[2010]{20C15, 20C30, 20C33}

\thanks{The second-named author is partially supported by a grant from the National Science Foundation (Award No. DMS-1801156). The third-named author is partially supported by the Spanish Ministerio
de Educaci\'on y Ciencia Proyectos MTM2016-76196-P and the ICMAT Severo Ochoa
project SEV-2011-0087.}

\author{Eugenio Giannelli}
\address{Dipartimento di Matematica e Informatica, U. Dini,
V. Morgagni 67, Firenze, Italy.}
\email{eugenio.giannelli@unifi.it}

\author{A. A. Schaeffer Fry}
\address{Dept. Mathematical and Computer Sciences, MSU Denver, Denver, CO 80217, USA}
\email{aschaef6@msudenver.edu}

\author{Carolina Vallejo Rodr\'iguez}
\address{ICMAT, Campus Cantoblanco UAM, C/ Nicol\' as Cabrera, 13-15,
  28049 Madrid, Spain}
\email{carolina.vallejo@icmat.es}

\title[Characters of $\pi'$-degree]{Characters of $\pi'$-degree}

\date{\today}

\begin{abstract}
Let $G$ be a finite group and let $\pi$ be a set of primes. Write $\irrpi G$ for the set of irreducible characters of degree not divisible by any prime in $\pi$. We show that if $\pi$ contains at most two prime numbers and the only element in $\irrpi G$ is the principal character, then $G=1$.
\end{abstract}

\maketitle

\section*{Introduction}  

Let $G$ be a finite group and let $\pi$ be a set of primes.  Write $\irrpi G$ for the set of irreducible characters of degree not divisible by any prime in $\pi$. If $\pi=\{ p \}$, then we use the standard notation $\irrp G=\irrpi G$. The condition $\irrp G=\{ 1_G\}$ implies that $G=1$ in an elementary way. Indeed, in such situation we have that $G$ is a $p'$-group, since the order of $G$ is the sum of the squares of the degrees of its irreducible characters. Hence $p$ does not divide the degree of any irreducible character of $G$ and $\irr G=\irrp G=\{1_G\}$ implies $G=1$ as wanted. We show that the same result holds if $\pi$ contains at most two primes.
\begin{thmA} Let $\pi=\{ p, q\}$ be a set of primes. If $\irrpi G=\{ 1_G\}$, then $G=1$.
\end{thmA}

We remark that the result no longer holds if $|\pi|>2$. For example, if $\pi=\{ 2, 3, 5  \}$ then, $\mathrm{Irr}_{\pi'}({\sf A}_{7})=\{1_{{\sf A}_7}\}$.  
%(Of course, also $\mathrm{Irr}_{\pi'}({\sf A}_{5})=\{1_{{\sf A}_5}\}$ but this might seem tricky as $\pi$ is exactly the set of primes dividing the order of ${\sf A}_5$.)

\smallskip

Often in Representation Theory of Finite Groups we find a duality between statements on irreducible characters and corresponding ones on conjugacy classes. 
For instance if $p$ is a prime and the conjugacy class of the identity is the unique conjugacy class of $p'$-size of $G$, then $G=1$. This is the dual statement of the one for irreducible characters described in the first paragraph of this section. 
We care to remark that the conjugacy class-version of Theorem A does not hold. 
For instance, the conjugacy class sizes of ${\sf A}_5$ are $1$, $15$, $20$, $12$ and $12$, so for every pair of primes $\pi$ dividing its order, the identity is the only conjugacy class of ${\sf A}_5$ of $\pi'$-size.

\smallskip

Our proof of Theorem A mainly relies on the Classification of Finite Simple Groups. 
We do not know if a CFSG-free proof might exist or if this result heavily depends on properties inherent to the representations of simple groups. 

\smallskip

The key observation to prove Theorem A is that for a simple group $G$ and any set $\pi=\{p, q\}$ of primes dividing the order of $G$, there exists some non-principal character in $\irrpi G$. 
Let $\Gamma'(G)$ be the undirected graph defined as follows.
The set of vertices of $\Gamma'(G)$ is the set of primes dividing the order of $G$, denoted $\pi(|G|)$. Two vertices $p$ and $q$ are adjacent if there is some $\chi\in\irr G \setminus {\rm Lin}(G)$ such that $\chi(1)$ is not divisible by $p$ nor by $q$. Here ${\rm Lin}(G)$ denotes the set of linear characters of $G$.  With this, the claim above can be stated in the following way.

\begin{thmB}\label{complete}
If $G$ is a non-abelian simple group, then $\Gamma' (G)$ is complete.
\end{thmB}

In fact, Theorem A implies that $\Gamma' (G)$ is complete for every perfect group $G$. We also analyze the opposite situation, namely, the case of finite groups $G$ with totally disconnected graph $\Gamma'(G)$.

\begin{thmC}
 Let $G$ be a group. Then $\Gamma'(G)$ is totally disconnected if, and only if, $G$ is solvable and $\norm G H \cap G'= H'$ for every $\pi$-Hall subgroup $H$ of $G$, where $\pi$ is any pair of primes dividing the order of $G$.
\end{thmC}

Surprisingly enough, if $\pi$ consists of two 
primes, then there are many examples where $\irrpi G={\rm Lin}(G)$. 
For instance, this is the case if $G={\rm PSL}_2(27)\cdot C_3$ and $\pi=\{3,13\}$.
Infinitely many other examples of this phenomenon can be found among symmetric, general linear, and general unitary groups as shown by Theorems \ref{teo: MainEu} and \ref{thm:complete_typeA} below.
In fact, we can precisely describe which groups satisfy this condition in the latter cases in Theorems \ref{teo: MainEu} and \ref{thm:mainLietype}. Finally, we characterize groups $G$ satisfying $\irrpi G={\rm Lin}(G)$ in terms of their solvable residual and discuss their non-abelian composition factors in Theorem \ref{solvable_residual} and Lemma \ref{lem:silly}.

%Abelian groups and groups of order $p^a q^b$ have totally disconnected graph $\Gamma'(G)$; but, of course, the condition is too restrictive and the converse of Theorem C does not hold. Among nilpotent groups, there are many counterexamples (see Remark \ref{nilpotent}). In \cite{N15}, Gabriel Navarro constructed a solvable group $G$ and a perfect group $H$ with $\pi(|G|)=\pi(|H|)$ and the same set of character degrees, so that $G$ is an example of a solvable (not nilpotent) group with complete $\Gamma' (G)$.
%Here we used the notation $\pi(n)$ to denote the set of prime divisors of the natural number $n$.
%It is worth mentioning that a group theoretical characterization of the condition $\irrpi G={\rm Lin}(G)$  in solvable groups for any set of primes is given in \cite{NW}.

%\smallskip

\smallskip

The paper is structured as follows. In Section 1 we prove Theorem A assuming that Theorem B holds. We also prove Theorem C, using previous results of Bianchi, Chillag, Lewis, and Pacifici \cite{BCLP} and of Navarro and Wolf \cite{NW}. The rest of the paper is mostly devoted to the proof of Theorem B on finite simple groups. In Section 2, we prove that $\Gamma'(G)$ is complete whenever $G$ is an alternating group, and we describe $\Gamma'(G)$ for symmetric groups. In Section 3, we prove that $\Gamma'(G)$ is complete when $G$ is a sporadic group or simple  group of Lie type, completing the proof of Theorem B by applying the Classification of Finite Simple Groups. We also provide there a description of $\Gamma'(G)$ when $G$ is a general linear or general unitary group. Finally, in Section 4, we discuss the structure of groups satisfying $\irrpi G={\rm Lin}(G)$ for a pair of primes $\pi$.

\smallskip

\noindent {\bf Acknowledgements.}
The authors would like to thank Gabriel Navarro for useful comments on a previous version of this paper and the anonymous referee for valuable suggestions. 
Finally, part of this work has been carried out while the third author was visiting the University of Florence. She thanks Carlo Casolo and the entire algebra group for the kind hospitality.

\section{On Theorems A and C}
Assuming that Theorem B holds, which follows from Corollary \ref{cor: EuMain} and Theorem \ref{thm:mainLietype} below, we can easily prove Theorem A.
\begin{proof}[Proof of Theorem A]
By way of contradiction assume that $G>1$. We may assume that $|\pi|=2$ and that $p$ and $q$ divide the order of $G$, otherwise the result follows from the case where $|\pi|=1$ treated in the introduction.  The fact that ${\rm Lin}(G)\sbs \irrpi G=\{ 1_G\}$ forces $G$ to be perfect. Moreover, if $N\nor G$ has index coprime to $p$, then $\irrpi {G/N}=\irrq {G/N}=\{1_{G/N}\}$ implies $N=G$. Similarly, one concludes that $G$ has no normal subgroup of index coprime to $q$. 
If we let $M\nor G$ be the first (proper) term in a composition series of $G$, then $S=G/M$ is a simple non-abelian group of order divisible by $p$ and $q$.  
Since the property is inherited by quotients of $G$, we have that $\mathrm{Irr}_{\pi'}(G/M)=\{1_{G/M}\}$.
By Theorem B we conclude that $G=M$, and this is a contradiction.
\end{proof}

The proof of Theorem C relies on  \cite{BCLP} and \cite{NW}. We will first show that if $\Gamma'(G)$ is totally disconnected, then the group $G$ must be solvable. For a group $G$, the common-divisor character degree graph $\Gamma(G)$ of $G$ is defined as follows.  The vertices of $\Gamma(G)$ are the degrees of the irreducible characters of $G$, and two vertices $a$ and $b$ are adjacent if $\gcd(a, b)>1$. In \cite{BCLP}, the authors prove that if $\Gamma(G)$ is complete, then $G$ is solvable.

\begin{lem}\label{BCLP} Let $G$ be a group with totally disconnected $\Gamma'(G)$. Then $G$ is solvable. 
\end{lem}
\begin{proof}By Burnside's $p^a q^b$-theorem we may assume that the order of $G$ is divisible by at least three different primes. Since $\Gamma' (G)$ is totally disconnected, if the order of $G$ is divisible by $m$ primes, then the degree of every non-linear irreducible character of $G$ is divisible by at least $m-1$ primes. 
In particular, $\Gamma(G)$ is complete. We conclude that $G$ is solvable by \cite[Main Theorem]{BCLP}.
\end{proof}

The condition $\irrpi G ={\rm Lin }(G)$ for solvable groups was studied in \cite{NW}. Notice that Theorem \ref{NW} below does not generally hold outside solvable groups (more precisely, outside $\pi$-separable groups) as Hall $\pi$-subgroups of $G$ may not exist. 

\begin{thm}[Navarro, Wolf]\label{NW} Let $G$ be a solvable group and let $\pi$ be any set of primes. Let $H$ be a Hall $\pi$-subgroup of $G$. Then $\irrpi G={\rm Lin} (G)$ if, and only if, $\norm G H\cap G'=H'$.
\end{thm}
\begin{proof}This is Corollary 3 in \cite{NW}.
\end{proof}

The aforementioned results allow us to characterize the groups $G$ with totally disconnected $\Gamma'(G)$. %Note that Lemma \ref{BCLP} will be key as Theorem \ref{NW} does not hold outside solvable groups. (For example, if $G={\sf S}_7$ for $\pi=\{ 2, 3\}$ and $H={\sf S}_3\times {\sf S}_4$ a self-normalizing $\pi$-Hall of $G$.)

\begin{proof}[Proof of Theorem C]
 If $\Gamma'(G)$ is totally disconnected, then by Lemma \ref{BCLP}, the group $G$ is solvable and the direct implication follows from Theorem \ref{NW}. The reverse implication follows directly from Theorem \ref{NW}.
\end{proof}

We end this section describing $\Gamma'(G)$ for nilpotent groups. Notice that nilpotent groups, and therefore solvable groups, can have complete graph $\Gamma'(G)$. In fact, examples of solvable, respectively nilpotent, groups with the same set of character degrees as a perfect group are provided in \cite{N15}, respectively \cite{NR}. 

\begin{rem}\label{nilpotent}Let $G$ be a nilpotent non-abelian group of order $|G|=p_1^{a_1}\cdots p_k^{a_k}$ for primes $p_i$, $a_i>0$ and $k\geq 3$. Recalling that $G$ is the direct  product of its Sylow subgroups, we see that $\Gamma'(G)$ is complete if, and only if, at most $k-3$ Sylow subgroups of $G$ are abelian. In the case where $k-2$ Sylow subgroups are abelian, there is an edge connecting every two primes except for the primes corresponding with the non-abelian Sylow subgroups. In the case where all but one Sylow, say the Sylow $p_k$-subgroup, are abelian, the subgraph of $\Gamma'(G)$ defined by $\{ p_1, \ldots, p_{k-1} \}$ is complete and the vertex $p_k$ is isolated. \end{rem}

\section{Alternating groups}\label{alternating}

The aim of this section is to prove Theorem B for alternating groups.

\subsection{Background}
We recall some basic facts in the representation theory of symmetric
groups. Standard references for this topic are \cite{James}, \cite{JK} and \cite{OlssonBook}.
A partition $\lambda=(\lambda_1,\lambda_2,\dots,\lambda_\ell)$ is a finite
non-increasing sequence of positive integers. If $n=\sum\lambda_i$, then we say that $\lambda$ is a
partition of $n$ and we write $\lambda\vdash n$ or, sometimes, $|\lambda|=n$. We denote by $\mathcal{P}(n)$ the set of partitions of $n$.
With a slight abuse of notation, given a sequence of partitions $T=(\mu_1,\ldots, \mu_t)$ we will write $|T|$ to denote the number $|\mu_1|+\cdots+|\mu_t|$. 

The Young diagram of a partition $\lambda$ is the set $$[\lambda]=\{(i,j)\in\NN\times\NN\mid
1\leq i\leq\ell,1\leq j\leq\lambda_i\},$$ where we orient $\NN\times\NN$ with
the $x$-axis pointing right and the $y$-axis pointing down. We denote by
$\lambda'$ the conjugate partition of $\lambda$, whose Young diagram is
obtained from that of $\lambda$ by a reflection over the main
diagonal.
Given $(r,c)\in[\lambda]$, the corresponding hook $H_{(r,c)}(\lambda)$ is the
set defined by
$H_{(r,c)}(\lambda)
  =\{(r,y)\in [\lambda]\mid y\geq c\}\cup\{(x,c)\in [\lambda]\mid x\geq r\}.$
We set $h_{r,c}(\lambda)=|H_{(r,c)}(\lambda)|=1+(\lambda_r-c)+(\lambda'_c-r)$. We refer to 
$h_{r,c}(\lambda)$ as the hook-length of $H_{(r,c)}(\lambda)$.
We denote by $\cH(\lambda)$ the multiset of hook-lengths in $[\lambda]$.
For $e\in\NN$ we let
$\cH^e(\lambda)=\{(r,c)\in [\lambda]\mid e \text{ divides }h_{r,c}(\lambda)\}$.
If $(r,c)\in\cH^e(\lambda)$, then we say that $H_{(r,c)}(\lambda)$ is an $e$-hook
of $\lambda$, so that $|\cH ^e (\lambda)|$ is the number of $e$-hooks of $\lambda$.
We record here an elementary observation (see \cite[Corollary 1.7]{OlssonBook}) that will be quite useful later in this section. 
\begin{lem}\label{lem: Olseasy}
Let $e,f\in\mathbb{N}$ and suppose that $h_{r,c}(\lambda)=ef$. Then 
$|\mathcal{H}^{e}(\lambda)\cap H_{(r,c)}(\lambda)|=f.$
\end{lem}

We denote by $\lambda - H_{(r,c)}(\lambda)$ the partition obtained by removing the $e$-hook $H_{(r,c)}(\lambda)$ from $\lambda$ (see \cite[Chapter I]{OlssonBook} for the precise definition of this process). 
The $e$-core $C_e(\la)$ of $\la$ is the partition obtained from $\la$ by
successively removing all $e$-hooks. The $e$-quotient $Q_e(\lambda)=(\lambda^{(0)}, \ldots, \lambda^{(e-1)})$ is another important combinatorial object, defined for instance in \cite[Section 3]{OlssonBook}.
The number of $e$-hooks to be removed from $\lambda$ to obtain $C_e(\lambda)$ is called the $e$-\textit{weight} $w_e(\lambda)$.
By \cite[3.6]{OlssonBook} we derive the following equations. 
\begin{equation}\label{Eq:E0}
|\lambda|=ew_e(\lambda)+|C_e(\lambda)|,\ \text{and}\ \ w_e(\lambda)=|\mathcal{H}^e(\lambda)|=|Q_e(\lambda)|=|\lambda^{(0)}|+\cdots+|\lambda^{(e-1)}|.
\end{equation}
Let $T_0^Q(\lambda)=(\lambda)$, $T_1^Q(\lambda)=Q_e(\lambda)=(\lambda^{(0)}, \ldots, \lambda^{(e-1)})$ and for $k\geq 1$ we define $T_{k+1}^Q(\lambda)$ to be the sequence of $e^{k+1}$ partitions given by 
$T_{k+1}^Q(\lambda)=((\lambda^{(i_1,\ldots, i_k)})^{(0)}, \ldots, (\lambda^{(i_1,\ldots, i_k)})^{(e-1)}),$
where $(i_1,\ldots, i_k)\in \{0,1,\ldots, e-1\}^k$. The collection of all the sequences $T_j^Q(\lambda)$ for $j\geq 0$ is known as the $e$-\textit{quotient tower} of $\lambda$. 
It is not too difficult to see that $|Q_{e^k}(\lambda)|=|T_k^Q(\lambda)|$, for all $k\in\mathbb{N}$. If $T_k^Q(\lambda)=(\mu_1,\ldots,\mu_{e^k})$ then we let $T^C_k(\lambda)=(C_e(\mu_1),\ldots,C_e(\mu_{e^k}))$. The collection of all the sequences $T_j^C(\lambda)$ for $j\geq 0$ is known as the $e$-\textit{core tower} of $\lambda$.

Let now $p$ be a prime. As shown in \cite[Chap.~II]{OlssonBook}, every partition of a given natural number is uniquely determined by its
$p$-core tower. Using the definitions given above we observe that for all $k\in\mathbb{N}_0$ we have that

\begin{equation}\label{Eq:E1}
|\mathcal{H}^{p^k}(\lambda)|=|T^Q_k(\lambda)|=\sum_{j\geq k}|T^C_j(\lambda)|p^{j-k},\ \ \text{in particular}\ \ |\lambda|=\sum_{j\geq 0}|T^C_j(\la)|p^j.
\end{equation}

%Note that $|T_0(\la)|=|C_q(\la)|$. Also note that if
%$n=\sum_{j=0}^k\alpha_jq^j$ is the $q$-adic expansion of $n$, then
%$|T_s(\la)|=0$ whenever $s>k$.
%It is useful to remark at this point that if $m$ is the maximal integer such
%that $|T_m(\la)|\neq 0$, then we have that $|\cH^{q^m}(\la)|=|T_m(\la)|$.

Partitions of $n$ correspond canonically to the irreducible characters of
${\sf S}_n$. We denote by $\chi^\la$ the irreducible character naturally labelled
by $\la\vdash n$. 
We use the notation $\lambda\vdash_{p'}n$ to say that $\chi^\lambda\in\mathrm{Irr}_{p'}({\sf S}_n)$.
We recall that $(\chi^\la)_{{\sf A}_n}$ is irreducible if,
and only if, $\la\neq \la'$. Otherwise $(\chi^\la)_{{\sf A}_n}=\phi+\phi^{g}$ for
some $\phi\in\Irr({\sf A}_n)$ and $g\in{\sf S}_n\smallsetminus{\sf A}_n$.
(See \cite[Thm.~2.5.7]{JK}.)
The following result was first proved by MacDonald \cite{Mac} and it is
crucial for our purposes. 

\begin{thm}  \label{theo: mac}
 Let $p$ be a prime and let $n$ be a natural number with $p$-adic expansion
 $n=\sum_{j=0}^ka_jp^j$. Let $\la$ be a partition of $n$. Then
 $$\nu_p(\chi^\la(1))
   =\big(\sum_{j\geq 0}|T^C_j(\la)|-\sum_{j=0}^ka_j\big)/(p-1).$$
Moreover, $\nu_p(\chi^\la(1))=0$ if, and only if, $|T^C_j(\la)|=a_j$ for all $j\in\mathbb{N}_{0}$.
\end{thm}
Here for a natural number $m$, we denoted by $\nu_p(m)$ the maximal
integer $k$ such that $p^k$ divides $m$. We will keep this notation for the rest of the article. 

A useful consequence of Theorem \ref{theo: mac} is stated in the following lemma. This is well-known to experts in the field.  
For the reader's convenience, we only include a brief proof. 
%of the implication we are going to use later in this section. A complete proof can be found in \cite{Mac}.

\begin{lem}   \label{lem: last layer}
 Let $p, n$ and $\la$ be as in Theorem~\ref{theo: mac}. Then $\la\vdash_{p'}n$ if, and only if, $|\cH^{p^k}(\la)|=a_k$ and $C_{p^k}(\lambda)\vdash_{p'}n-a_kp^k$. 
\end{lem}

\begin{proof}
Assume that $|\cH^{p^k}(\la)|=a_k$ and that $C_{p^k}(\lambda)\vdash_{p'}n-a_kp^k$.
Since $p^{k+1}>n$, it follows that $|T^C_{s}(\lambda)|=0$ for all $s>k$. Hence using equation \eqref{Eq:E1} above, we get that $a_k=|\mathcal{H}^{p^k}(\lambda)|=|T^C_k(\lambda)|$. By \cite[Theorem 3.3]{OlssonBook} one deduces that $T_j^C(C_{p^k}(\lambda))=T_j^C(\lambda)$ for all $0\leq j<k$. Using Theorem \ref{theo: mac}, we deduce that $\la\vdash_{p'}n$. The converse implication is a direct consequence of Theorem \ref{theo: mac}.
\end{proof}

Let $\mathcal{L}(n):=\{(n-x,1^x)\ |\ 0\leq x\leq n-1\}$ be the set all the \textit{hook partitions} of the natural number $n$. The following fact is an immediate consequence of Lemma \ref{lem: last layer}.

\begin{lem}\label{hooks}
Let $k\in\mathbb{N}$ and let $p$ be a prime. Then $\chi^\lambda\in\mathrm{Irr}_{p'}({\sf S}_{p^k})$ if, and only if, $\lambda\in\mathcal{L}(p^k)$. 
\end{lem}

%\begin{proof}
%It is not difficult to observe that $T_j(C_{q^k}(\lambda))=T_j(\lambda)$ for all $0\leq j<k$ and that $|T_kC_{q^k}(\lambda))|=0$. The statement now follows from Theorem \ref{theo: mac}.
%\end{proof}

%\begin{lem}   \label{lem: k+1-alpha}
% Let $k\geq 2$ and let $n=2^{k+1}-2^\ell$, for some $\ell\leq k-2$. Let
% $\la\vdash n$ such that $|\cH^{2^k}(\la)|=0$ and $|\cH^{2^{k-1}}(\la)|\leq 
%2$.
% Then $\nu_2(\chi^\la(1))\geq 2$.
%\end{lem}

%\begin{proof}
%Since $|\cH^{2^k}(\la)|=0$ we deduce that $|T_k(\la)|=0$. Hence we have that
%$|T_{k-1}(\la)|=|\cH^{2^{k-1}}(\la)|\leq 2$.
%We know that $n=\sum_{j=0}^{k}|T_j(\la)|2^j$, hence we obtain that
%$\sum_{j=0}^{k}|T_j(\la)|\geq k+3-\ell$. Since $n$ has $k+1-\ell$ binary
%digits, we deduce that $\nu_2(\chi^\la(1))\geq 2$ from Theorem~\ref{theo: 
%mac}.
%\end{proof}

%We conclude our preliminaries with the following technical lemma. 

\begin{lem}\label{lem: oldlayer}
Let $k\in\mathbb{N}_{>0}$ and let $\varepsilon\in\{0,1\}$. Let $n=2^k+\varepsilon$ and let $\lambda\vdash n$. If $\mathcal{H}^{2^{k}}(\lambda)=\emptyset$ and $|\mathcal{H}^{2^{k-1}}(\lambda)|=2$ then $\nu_2(\chi^\lambda(1))=1$. 
\end{lem}
\begin{proof}
By equations \eqref{Eq:E0} we have that $n=|C_{2^{k-1}}(\lambda)|+2^{k-1}|\mathcal{H}^{2^{k-1}}(\lambda)|$. It follows that $|C_{2^{k-1}}(\lambda)|=\varepsilon$. In particular $C_{2^{k-1}}(\lambda)$ is a $2$-core partition. 
Since removing a $2^{k-1}$-hook does not change the $2$-core of a partition, we have that $C_2(\lambda)=C_2(C_{2^{k-1}}(\lambda))\vdash\varepsilon$. Thus
$|T^C_0(\lambda)|=|C_2(\lambda)|=\varepsilon$. 
Moreover, using equations \eqref{Eq:E1}, we see that $|T^C_j(\lambda)|=0$ for all $j\geq k$. Hence, again by equations \eqref{Eq:E1} we conclude that
$$|T^C_{k-1}(\lambda)|=|T^Q_{k-1}(\lambda)|=|\mathcal{H}^{2^{k-1}}(\lambda)|=2,\ \text{and hence}\ |T^C_{j}(\lambda)|=0\ \text{for all}\ 1\leq j\leq k-2.$$
It follows that $\sum_{j\geq 0}|T_j^C(\la)|=2+\varepsilon$, while $\sum_{j=0}^ka_j=1+\varepsilon$. The statement now follows from Theorem \ref{theo: mac}.
\end{proof}

%Finally, we state the following consequence of Nakayama's Conjecture (proved
%independently by R.~Brauer and G.~de B.~Robinson, see e.g. \cite{James} or
%\cite{OlssonBook}).
%%%\cite{BrauerNakayama} and \cite{RobinsonNakayama}).
%The second statement follows from \cite[Thm.~9.2]{Nav98}.
%
%\begin{prop}   \label{prop: Nak}
% Let $p$ be a prime and let $n=a+pw$ for some $a\in\{0,1,\ldots, p-1\}$ and
% $w\in\NN$. Let $\la\vdash n$. Then $\chi^\la$ lies in the principal $p$-block
% of ${\sf S}_n$ if, and only if, $C_p(\la)=(a)$. Moreover every irreducible
% constituent of $(\chi^\la)_{{\sf A}_n}$ lies in the principal $p$-block of 
%${\sf A}_n$.
%\end{prop}

\subsection{Main Results}
We are now ready to prove Theorem B for alternating groups. 
%In order to do this we will first completely describe $\Gamma'({\sf S}_n)$. 

\begin{thm}   \label{teo: MainEu}
 Let $n\geq 5$ and let  $q$ and $p$ be distinct primes such that $2\leq q, p\leq n$. Let $\pi=\{p,q\}$.
Then $\mathrm{Irr}_{\pi'}({\sf S}_n)=\mathrm{Lin}({\sf S}_n)$ if, and only if, 
$q=2$ and 
$$n=\begin{cases} 2^k=p^m+1 & \mathrm{or},\\
2^{k}+1=p^m.\end{cases}$$
\end{thm}

\begin{proof}
We aim to construct a partition $\lambda\in\mathcal{P}(n)\smallsetminus\{(n),(1^n)\}$ such that $\chi^\lambda\in\mathrm{Irr}_{\pi'}({\sf S}_n)$.
%$(\phi(1), pq)=1$ for any $\phi\in\mathrm{Irr}({\sf A}_n)$ such that $\phi$ is an irreducible constituent of $(\chi^\lambda)_{{\sf A}_n}$.
Let $$\sum_{i=1}^ta_ip^{m_i}=n=\sum_{i=1}^rb_iq^{k_i},$$
be the $p$-adic and respectively the $q$-adic expansions of $n$, where $m_1>m_2>\cdots m_t\geq 0$ and $k_1>k_2>\cdots k_r\geq 0$. 
It is clear that $b_1q^{k_1}\neq a_1p^{m_1}$.

Suppose that $b_1q^{k_1}< a_1p^{m_1}$ (the statement in the opposite setting will be proved just by swapping $p$ with $q$, $a_1$ with $b_1$ and $m_1$ with $k_1$ in what follows). 
Consider the partition $\lambda$ of $n$ defined by:
$$\lambda=(n-b_1q^{k_1}, n-a_1p^{m_1}+1, 1^{b_1q^{k_1}-(n-a_1p^{n_1}+1)}).$$
Observe that $\lambda$ is a well-defined partition as $\lambda_1\geq \lambda_2\geq\lambda_j=1$, for all $j\geq 3$.
Since $h_{1,1}(\lambda)=a_1p^{m_1}$, we deduce that $|\mathcal{H}^{p^{m_1}}(\lambda)|\geq a_1$ from Lemma \ref{lem: Olseasy}. By equations \eqref{Eq:E0} we know that $|\mathcal{H}^{p^{m_1}}(\lambda)|=w_{p^{m_1}}(\lambda)\leq a_1$. It follows that $|\mathcal{H}^{p^{m_1}}(\lambda)|=a_1$.
Moreover $$C_{p^{m_1}}(\lambda)=\lambda - H_{1,1}(\lambda)=(n-a_1p^{m_1})\vdash_{p'}n-a_1p^{m_1}.$$ Using Lemma \ref{lem: last layer}, we conclude that $\chi^\lambda\in\mathrm{Irr}_{p'}({\sf S}_n)$. 
On the other hand, using a similar argument, we can also show that $\chi^\lambda\in\mathrm{Irr}_{q'}({\sf S}_n)$. This follows from Lemma \ref{lem: last layer} because $h_{2,1}(\lambda)=b_1q^{k_1}$, hence $|\mathcal{H}^{q^{k_1}}(\lambda)|=w_{q^{k_1}}(\lambda)=b_1$, and again it is routine to check that $C_{q^{k_1}}(\lambda)=(n-b_1q^{k_1})\vdash_{q'}n-b_1q^{k_1}$.
We obtain that $\chi^\lambda\in\mathrm{Irr}_{\pi'}({\sf S}_n)$. To conclude we need to make sure that $\chi^\lambda\notin\mathrm{Lin}({\sf S}_n)$ (i.e. $\lambda\notin\{(n), (1^n)\}$). 
%If $\phi\in\mathrm{Irr}({\sf A}_n)$ is a constituent of $(\chi^\lambda)_{{\sf A}_n}$ then clearly $(\phi(1), pq)=1$. 
%To conclude we need to make sure that $\phi$ is not trivial (i.e. $\lambda\notin\{(n), (1^n)\}$). 

Since $\lambda_2\geq 1$ we notice that $\lambda\in\{(n), (1^n)\}$ if ,and only if, $\lambda=(1^n)$ and this happens if, and only if, $n-a_1p^{m_1}=\lambda_2-1=0$ and $n-b_1q^{k_1}=\lambda_1=1$. Equivalently, $\lambda\in\{(n), (1^n)\}$ if, and only if,
$$n=ap^{m}=bq^{k}+1.$$
(To ease the notation we renamed $a_1, m_1, b_1$ and $k_1$ by $a, m, b$ and $k$ respectively.) 
 In this very specific situation $\lambda=(1^n)$,
%and $\phi(1)=1$. 
hence we need to pick a different partition. In order to make this new choice we distinguish two main cases, depending on $2\in\{p,q\}$ or not. 

Let us first assume that $2\notin\{p,q\}$. 
If $b\geq 2$ then let $\mu\vdash n$ be defined as follows: $$\mu=(1+(b-1)q^{k}, 1^{q^{k}}).$$
Notice that $h_{1,1}(\mu)=ap^{m}$, $h_{1,2}(\mu)=(b-1)q^{k}$ and $h_{2,1}(\mu)=q^{k}$. Hence from Lemma \ref{lem: Olseasy} we obtain that $|\mathcal{H}^{p^{m}}(\mu)|=a$ and that $|\mathcal{H}^{q^{k}}(\mu)|=b$. Moreover, $C_{p^{m}}(\mu)=\emptyset$ and $C_{q^{k}}(\mu)=(1)$. Using Lemma \ref{lem: last layer} we deduce that $\chi^\mu\in\mathrm{Irr}_{\pi'}({\sf S}_n)$. 
Otherwise if $b=1$, then  $a=2c$ is even and we let 
$$\mu=(cp^{m}, 2, 1^{cp^{m}-2}).$$
It is now routine to check that $h_{1,1}(\mu)=q^{k}$, $h_{1,2}(\mu)=h_{2,1}(\mu)=cp^{m}$, $|\mathcal{H}^{p^{m}}(\mu)|=a$, $|\mathcal{H}^{q^{k}}(\mu)|=1$, $C_{p^{m}}(\mu)=\emptyset$ and $C_{q^{k}}(\mu)=(1)$.
Lemma \ref{lem: last layer} implies that $\chi^\mu\in\mathrm{Irr}_{\pi'}({\sf S}_n)$.  
%In both the above cases we have that if $\phi$ is an irreducible constituent of $(\chi^\mu)_{{\sf A}_n}$ then $\phi\in\mathrm{Irr}_{p'}({\sf A}_n)\cap\mathrm{Irr}_{q'}({\sf A}_n)$.

\medskip

To conclude we now analyze the case where $2\in\{q,p\}$. 
%Setting $q=2$ 
We have two substantially different cases to consider. 
Namely $n=ap^{m}=2^{k}+1$ and $n=ap^{m}+1=2^{k}$.

\smallskip

\textbf{-} If $n=2^k=ap^m+1$ then $\chi^\mu\in\mathrm{Irr}_{2'}({\sf S}_n)$ if, and only if, $\mu\in\mathcal{L}(n)$, by Lemma \ref{hooks}. 

\noindent If $a>1$ then we let $\mu$ be the partition of $n$ defined as follows.
$$\mu=(1+(a-1)p^m, 1^{p^m})\in\mathcal{L}(n)\smallsetminus\{(n), (1^n)\}.$$ 
We observe that $h_{1,2}(\mu)=(a-1)p^m$ and $h_{2,1}(\mu)=p^m$. Hence $|\mathcal{H}^{p^{m}}(\mu)|=a$ by Lemma \ref{lem: Olseasy}, and $C_{p^{m}}(\mu)=(1)$. Therefore $\chi^\mu\in\mathrm{Irr}_{\{2,p\}'}({\sf S}_n)$ by Lemma \ref{lem: last layer}.

On the other hand, if $a=1$ then it is not difficult to see that the only characters in $\mathrm{Irr}_{p'}({\sf S}_n)$ labelled by hook partitions are the trivial and the sign character. More precisely, we have that $\mathrm{Irr}_{p'}({\sf S}_n)\smallsetminus\mathrm{Lin}({\sf S}_n)=\{(n-x,2,1^{x-2})\ |\ 2\leq x\leq n-2\}.$
Hence $\mathrm{Irr}_{\{2,p\}'}({\sf S}_n)=\mathrm{Lin}({\sf S}_n)$. 

\smallskip

\textbf{-} If $n=2^k+1=ap^m$ and $a>1$ then we define $\mu\vdash n$ as follows. 
$$\mu=(1+(a-1)p^m, 2, 1^{p^m-2}).$$ It is routine to check that $h_{11}(\mu)=2^k$, $h_{1,2}(\mu)=(a-1)p^m$ and that $h_{2,1}(\mu)=p^m$. We conclude that $\chi^\mu\in\mathrm{Irr}_{\{2,p\}'}({\sf S}_n)$, by Lemma \ref{lem: last layer}.

If $a=1$ then $\chi^\mu\in\mathrm{Irr}_{p'}({\sf S}_n)$ if, and only if, $\mu\in\mathcal{L}(n)$, by Lemma \ref{hooks}. 
It is now easy to check
that the only characters in $\mathrm{Irr}_{2'}({\sf S}_n)$ labelled by hook partitions are the trivial and the sign character. 
Hence $\mathrm{Irr}_{\{2,p\}'}({\sf S}_n)=\mathrm{Lin}({\sf S}_n)$. 
\end{proof}

%\begin{rem}\label{catalan}
%We find interesting to observe that applying Catalan's conjecture (proved in \cite{Mih04}) one can restate Theorem \ref{teo: MainEu} above by saying that $\mathrm{Irr}_{\pi'}({\sf S}_n)=\mathrm{Lin}({\sf S}_n)$ if, and only if, 
%$q=2$ and 
%$$n=\begin{cases} p+1 & \text{ for some Mersenne prime $p$,}\\
%p &\text{ for some Fermat prime $p$,}\\
%9 & \text{ for $p=3$}.
%\end{cases}$$
%\end{rem}

%\begin{rem}\label{Hall} Symmetric groups that satisfy the condition $\mathrm{Irr}_{\pi'}({\sf S}_n)=\mathrm{Lin}({\sf S}_n)$ do not possess Hall $\pi$-subgroups. MUST BE EXPANDED
%\end{rem}

Theorem \ref{teo: MainEu} can be reformulated as follows. 

\begin{cor}\label{complete_symmetric}
The graph $\Gamma'({\sf S}_n)$ is not complete if, and only if, 
$n=2^k=p^m+1$ or $n=2^{k}+1=p^m$. 
In these cases the subgraphs of $\Gamma'({\sf S}_n)$ defined by $\pi(|{\sf S}_n|)\smallsetminus\{2\}$ and by $\pi(|{\sf S}_n|)\smallsetminus\{p\}$ are complete, and there is no edge between $2$ and $p$. 
\end{cor}

%\begin{rem}
%The distinction between the various cases in the proof of Theorem \ref{teo: Sym q} is non avoidable. For instance observe that if $n=2^3=7+1$, then 
%$\mathrm{Irr}_{2'}({\sf S}_8)\cap\mathrm{Irr}_{7'}({\sf S}_8)=\mathrm{Lin}({\sf S}_8)$. Nevertheless, if $\theta\in\mathrm{Irr}({\sf S}_8)$ is labelled by the partition $(4,2,1^2)$, then the two irreducible non-linear constituents of $\theta_{{\sf A}_8}$ lie in $\mathrm{Irr}_{2'}({\sf A}_8)\cap\mathrm{Irr}_{7'}({\sf A}_8)$.
%\end{rem}

\begin{cor}\label{cor: EuMain}
Let $p\neq q$ be two primes and define $\pi:=\{p,q\}$. Then
$|\mathrm{Irr}_{\pi'}({\sf A}_n)|>1$, for all $n\geq 5$.
\end{cor}
\begin{proof}
By Theorem \ref{teo: MainEu} there exists $\lambda\in\mathcal{P}(n)\smallsetminus\{(n), (1^n)\}$ such that $\chi^\lambda\in\mathrm{Irr}_{\pi'}({\sf S}_n)$, unless $q=2$ and $n=2^k=p^m+1$ or $n=2^k+1=p^m$.
In all these cases, let $\phi\in\mathrm{Irr}({\sf A}_n)$ be an irreducible constituent of $(\chi^\lambda)_{{\sf A}_n}$. Then $\phi(1)\in\{\chi^\lambda(1), \chi^\lambda(1)/2\}$. In particular $\phi\in\mathrm{Irr}_{\pi'}({\sf A}_n)$. 
Suppose now that $q=2$ and $n=2^k=p^m+1$ or $n=2^k+1=p^m$. 
We let $\mu$ be the partition of $n$ defined as follows: 
$$\mu=\begin{cases} (2^{k-1}+1, 1^{2^{k-1}}) & \mathrm{if}\ n=p^{m}=2^{k}+1,\\
\\
(2^{k-1}, 2, 1^{2^{k-1}-2}) & \mathrm{if}\ n=p^{m}+1=2^{k}.\end{cases}$$
%If $n=ap^{m}=2^{k}+1$ then we let $\mu\vdash n$ be defined as follows:  $$\mu=(2^{k-1}+1, 1^{2^{k-1}}).$$ 
%Otherwise, if $n=aq^{m}+1=2^{k}$ then we let $\mu\vdash n$ be defined as follows: $$\mu=(2^{k-1}, 2, 1^{2^{k-1}-2}).$$ 
(Notice that $\mu$ is well defined because $n>4$.)
In both cases we see that $\mathcal{H}^{2^{k}}(\mu)=\emptyset$ and that $|\mathcal{H}^{2^{k-1}}(\mu)|=2$. Therefore by Lemma \ref{lem: oldlayer} we deduce that $\nu_2(\chi^\mu(1))=1$. Since $\mu=\mu'$ we conclude that $\phi(1)=\chi^{\mu}(1)/2$ is odd, for any irreducible constituent $\phi$ of $(\chi^\mu)_{{\sf A}_n}$. Moreover, in both cases $h_{1,1}(\mu)=p^{m}$ and 
$$C_{p^{m}}(\mu)=\begin{cases} \emptyset & \mathrm{if}\ n=p^{m}=2^{k}+1,\\

(1) & \mathrm{if}\ n=p^{m}+1=2^{k}.\end{cases}$$
Hence Lemma \ref{lem: last layer} guarantees that $\chi^\mu\in\mathrm{Irr}_{p'}({\sf S}_n)$. We conclude that $\phi\in\mathrm{Irr}_{\pi'}({\sf A}_n)$.
\end{proof}

We conclude this section by discussing the extendibility to ${\sf S}_n$ of characters in $\mathrm{Irr}_{\pi'}({\sf A}_n)$.
The ideas used in the proof of Theorem \ref{teo: MainEu} can be adapted to classify those alternating groups whose only extendible $\pi'$-character is the trivial character. As we will remark in Section 4, the following result will be useful to study the normal structure of finite groups $G$ such that $\mathrm{Irr}_{\pi'}(G)=\mathrm{Lin}(G)$.

\begin{prop}\label{prop:Eunew}
Let $\pi=\{ p,q\}$ where $p$ and $q$ are distinct primes and let $n\geq 5$ be a natural number. Then ${\sf A}_n$ admits a non-trivial irreducible character of $\pi'$-degree that extends to ${\sf S}_n$, unless 
%$$n=\begin{cases} 2^k=p^m+1 &\mathrm{or},\\
%2^{k}+1=p^m.\end{cases}
%\ \ \text{or}\ \ p,q\ \text{are odd and}\ \ 
$$n=\begin{cases} 2q^k=p^m+1 & \mathrm{or},\\
2q^{k}+1=p^m.\end{cases}
$$
\end{prop}
\begin{proof}
If $q=2$ and $n=2^k=p^m+1$  or $n=2^k+1=p^m$ for some $k,m\in\mathbb{N}$, then $\mathrm{Lin}({\sf S}_n)=\mathrm{Irr}_{\pi'}({\sf S}_n)$ by Theorem \ref{teo: MainEu}. It follows that ${\sf A}_n$ admits no non-trivial irreducible character of $\pi'$-degree that extends to ${\sf S}_n$. If $q$ is odd and $n=2q^{k}+1=p^m$ then using Lemma \ref{hooks} we see that the only non-linear irreducible $\pi'$-degree character of ${\sf S}_n$ is labelled by the partition $\lambda=(1+q\6k, 1\6{q\6k})$. Similarly, if $n=2q^{k}=p^m+1$ then the only non-linear irreducible $\pi'$-degree character of ${\sf S}_n$ is labelled by the partition $\lambda=(q\6k, 2, 1\6{q\6k-2})$. In both cases $\lambda=\lambda'$ and therefore we deduce that ${\sf A}_n$ admits no non-trivial irreducible character of $\pi'$-degree that extends to ${\sf S}_n$, by \cite[Thm. 2.5.7]{JK}.

We now recycle the notation used in the proof of Theorem \ref{teo: MainEu} and let $$\sum_{i=1}^ta_ip^{m_i}=n=\sum_{i=1}^rb_iq^{k_i},$$
be the $p$-adic and respectively the $q$-adic expansions of $n$. Again we assume that $b_1q^{k_1}< a_1p^{m_1}$ and we consider $\lambda=(n-b_1q^{k_1}, n-a_1p^{m_1}+1, 1^{b_1q^{k_1}-(n-a_1p^{n_1}+1)}).$
As shown in the proof of Theorem \ref{teo: MainEu} we have that $\chi\6\lambda\in\mathrm{Irr}_{\pi'}({\sf S}_n)$. Moreover, if $n-a_1p\6{m_1}\geq 2$ then $\lambda_2\geq 3$ and hence $\lambda\neq\lambda'$. It follows that $(\chi\6\lambda)_{{\sf A}_n}$ is a non-trivial irreducible character of $\pi'$-degree that extends to ${\sf S}_n$.
We are left to consider the cases where $n-a_1p\6{m_1}\leq 1$. If $n-a_1p\6{m_1}=1$ then $\lambda\neq\lambda'$ because $(\lambda')_1=b_1q\6{k_1}>n-b_1q\6{k_1}=\lambda_1$. Moreover $\lambda\notin\{(n), (1\6n)\}$ because $\lambda_2=2$. It follows that also in this case $(\chi\6\lambda)_{{\sf A}_n}$ is a non-trivial irreducible character of $\pi'$-degree that extends to ${\sf S}_n$. 
If $n-a_1p\6{m_1}=0$ then again $\lambda\neq\lambda'$ because $(\lambda')_1=b_1q\6{k_1}+1>n-b_1q\6{k_1}=\lambda_1$. Hence $(\chi\6\lambda)_{{\sf A}_n}$ is a non-trivial irreducible character of $\pi'$-degree that extends to ${\sf S}_n$, unless $\lambda\in\{(n), (1\6n)\}$. This happens if, and only if, 
$n=ap\6{m}=bq\6{k}+1$. (To ease the notation we renamed $a_1, m_1, b_1$ and $k_1$ by $a, m, b$ and $k$ respectively).
Suppose first that $2\notin\{p,q\}$. From the first part of this proof, we can assume that $(a,b)\notin\{(1,2), (2,1)\}$.
If $b\geq 3$ and $\mu=(1+(b-1)q^{k}, 1^{q^{k}}),$ then arguing in a very similar fashion as in the proof of Theorem \ref{teo: MainEu} we observe that
$(\chi\6\mu)_{{\sf A}_n}$ is a non-trivial irreducible character of $\pi'$-degree that extends to ${\sf S}_n$. 
If $b\leq 2$, we let $\nu=((a-1)p\6m, 2, 1\6{p\6m-2})$. This is a well defined partition as $a>1$. This can be deduced by noticing that when $b=1$ then $a$ must be even, and for $b=2$ then $(a,b)\neq (1,2)$. 
 By Lemma \ref{lem: last layer} we see that $\chi\6\nu\in\mathrm{Irr}_{\pi'}({\sf S}_n)$. Moreover, we notice that if $b=2$ then $h_{1,1}(\nu)=2q\6k$ is even and hence $\nu\neq \nu'$. On the other hand, if $b=1$ then $a$ is even and strictly greater than $2$ (as $(a,b)\neq (2,1)$). Thus also in this case it follows that $\nu\neq\nu'$. We conclude that 
$(\chi\6\nu)_{{\sf A}_n}$ is a non-trivial irreducible character of $\pi'$-degree that extends to ${\sf S}_n$.
%and $a>1$ then we let $\nu=((a-1)p\6m, 2, 1\6{p\6m-2})$ and observe that $(\chi\6\nu)_{{\sf A}_n}$ is a non-trivial irreducible character of $\pi'$-degree that extends to ${\sf S}_n$.
%If $b=1$ then $a$ is even and strictly greater than $2$. In this case we consider 
%$\zeta=((a-1)p\6m, 2, 1\6{p\6m-2})$. Since $a>2$, it follows that $\zeta\neq\zeta'$. By Lemma \ref{lem: last layer} we deduce that $\chi\6\zeta\in\mathrm{Irr}_{\pi'}({\sf S}_n)$ and therefore that 
%$(\chi\6\zeta)_{{\sf A}_n}$ is a non-trivial irreducible character of $\pi'$-degree that extends to ${\sf S}_n$.
Finally, we study the case where $2\in\{p,q\}$. If $n=2\6k=ap\6m+1$ then we can assume that $a\geq 2$, from the first part of the proof. In this case we take $\mu=(1+(a-1)p^m, 1^{p^m})$ as in the proof of Theorem \ref{teo: MainEu}.
Similarly, if $n=2\6k+1=ap\6m$ then $a$ is odd and again we can use the first part of the proof to assume that $a\geq 3$ and we let $\mu=(1+(a-1)p^m, 2, 1^{p^m-2})$.
In both cases 
we observe that 
$(\chi\6\mu)_{{\sf A}_n}$ is a non-trivial irreducible character of $\pi'$-degree that extends to ${\sf S}_n$.
\end{proof}

\section{Groups of Lie Type and Sporadic Groups}

The aim of this section is to complete the proof of Theorem B. 
We begin with the case of the sporadic groups and certain groups of Lie type that may be treated computationally.

\begin{lem}\label{lem:lowrank}
The simple groups %$A_1(3), A_1(2), \tw{2}A_2(2), B_2(2), G_2(2)$,
$G_2(3)$, $G_2(4)$,  $\tw{2}G_2(3)'$, $\tw{2}F_4(2)',$ and $\tw{2}E_6(2)$  and the 26 sporadic groups satisfy Theorem B.
\end{lem}
\begin{proof}
This can be seen using GAP and the Character Table Library.
\end{proof}

Let $G=\bG^F$ be the group of fixed points of a connected reductive algebraic group $\bG$ defined over $\overline{\mathbb{F}}_r$ under a Steinberg map $F$. Here $r$ is a prime and $\overline{\mathbb{F}}_r$ is an algebraic closure of the finite field of cardinality $r$. We will call a group $G$ of this form a finite reductive group.  Further, let $G^\ast=(\bG^\ast)^{F^\ast}$, where $(\bG^\ast, F^\ast)$ is dual to $(\bG, F)$.   

The set of irreducible characters $\irr G $ can be written as a disjoint union $\bigsqcup \mathcal{E}(G, s)$ of rational Lusztig series corresponding to $G^\ast$-conjugacy classes of semisimple elements $s\in G^\ast$.  The characters in the series $\mathcal{E}(G, 1)$ are called unipotent characters, and there is a bijection $\mathcal{E}(G,s)\rightarrow\mathcal{E}(C_{G^\ast}(s), 1)$.  Hence, characters of $\irr G$ may be indexed by pairs $(s, \psi)$, where $s\in G^\ast$ is a semisimple element, up to $G^\ast$-conjugacy, and $\psi\in\irr {C_{G^\ast}(s)}$ is a unipotent character.  

Further, if $\chi\in\irr G$ is indexed by the pair $(s,\psi)$, then the degrees of $\chi$ and $\psi$ are related in the following way: 
\begin{equation}\label{DL}
\chi(1)=|G^\ast:C_{G^\ast}(s)|_{r'}\psi(1)
\end{equation}
(see \cite[Remark 13.24]{dignemichel}).  Here for a natural number $n$ and a prime $p$, we denote by $n_{p'}$ the largest integer $m$ dividing $n$ such that $\gcd(m,p)=1$.  Similarly, we will denote the number $p^{\nu_p(n)}$ by $n_p$.

From \eqref{DL}, we immediately see the following:

\begin{prop}\label{prop:descdl}
Let $G=\bG^F$ be a finite reductive group defined over a field of characteristic $r$.  Let $p\neq q$ be two primes and define $\pi:=\{p,q\}$.

\begin{itemize}
\item If $r\not\in\pi$, then the set $\irrpi G$ is parametrized by pairs $(s, \psi)$ as above, where $s$ centralizes both a Sylow $p$-subgroup and Sylow $q$-subgroup of $G^\ast$ and $\psi\in\irrpi {C_{G^\ast}(s)}$.
\item If $r=p$, then $\irrpi G$ is parametrized by pairs $(s,\psi)$ as above where $s$ centralizes a Sylow $q$-subgroup of $G^\ast$ and $\psi\in\irrpi {C_{G^\ast}(s)}$.  
\end{itemize}
\end{prop}
Moreover, \cite[Theorem 6.8]{mallehz} yields that in the second statement of Proposition \ref{prop:descdl}, we may often further say $\psi(1)=1$.

\begin{lem}\label{lem:reductive}
Let $G=\bG^F$ be a finite reductive group defined over a field of characteristic $r$.  Let $p\neq q$ be two primes and define $\pi:=\{p,q\}$.  Then $|\irrpi G|>1$.
\end{lem}
\begin{proof}
First, suppose $p\neq r$ and $q\neq r$.  Then the degree of the Steinberg character $\mathrm{St}_G$ is a power of $r$  (see for example \cite[Corollary 9.3]{dignemichel}), and is therefore an element in $\irrpi G$.  Then we may assume that $p$ is the defining characteristic for $G$.  That is, we assume $r=p$.

Now, let $Q$ be a Sylow $q$-subgroup of $G^\ast$ and let $s\in Z(Q)$ be nontrivial.  Then $s$ is semisimple, since $q\neq p$, and certainly $Q\leq C_{G^\ast}(s)$.  Hence taking $\chi$ to be indexed by $(s, 1_{C_{G^\ast}(s)})$, we have $q\nmid \chi(1)$ and $p\nmid \chi(1)$ from \eqref{DL}, so $\chi\in\irrpi G$.
\end{proof}

Before we extend the above result to prove Theorem B in the case of simple groups of Lie type, we note the following straightforward but useful lemma.

\begin{lem}\label{lem:centerkernel}
Let $G$ be a perfect group and let $q$ be prime.  Suppose that $|Z(G)|$ is a power of $q$ and that $\chi\in\irr{G}$ has degree prime to $q$.  Then $Z(G)$ is in the kernel of $\chi$.  
\end{lem}
\begin{proof}
Write $Z=Z(G)$. The order $o(\lambda)$ of $\lambda \in \irr Z$ lying under $\chi$ must divide $\chi(1)$, since $1_Z=\det{\chi}_Z=\lambda^{\chi(1)}$.\end{proof}
 
We are now ready to complete the proof of Theorem B.

\begin{thm}\label{thm:mainLietype}
Let $S$ be a finite simple group such that $S=G/Z(G)$ for $G=\bG^F$ a finite reductive group of simply connected type defined over $\mathbb{F}_{r^a}$ for some prime  $r$.  Let $p\neq q$ be two primes and define $\pi:=\{p,q\}$.  Then $|\irrpi S|>1$.
\end{thm}
\begin{proof}
By Lemma \ref{lem:lowrank}, we may assume $S$ is not one of the groups listed there, and by Section \ref{alternating}, we may assume $S$ is not isomorphic to an alternating group.  We wish to show that the character $\chi$ constructed in the proof of Lemma \ref{lem:reductive} can be chosen to be trivial on $Z(G)$.  This is satisfied for $\mathrm{St}_G$, so we again assume $r=p$. By \cite[Lemma 4.4(ii)]{NavarroTiep13}, it therefore suffices to show that the semisimple element $s$ used in the proof of Lemma \ref{lem:reductive} can be chosen to be contained in $(G^\ast)'$.

Here $G$ is quasisimple and $|Z(G)|=|G^\ast: (G^\ast)'|$, where $(G^\ast)'$ denotes the commutator subgroup of $G^\ast$. Now, if $q\nmid |Z(G)|$, then $Q\leq (G^\ast)'$, where $Q$ is any Sylow $q$-subgroup of $G^\ast$, and the $s$ chosen in Lemma \ref{lem:reductive} is in this case contained in $(G^\ast)'$.  Combining this with Lemma \ref{lem:centerkernel}, we may therefore assume that $q$ divides $|Z(G)|$ but that $|Z(G)|$ is not a power of $q$.  This leaves only the case that $\bG$ is of type $A$.

Hence, for the remainder of the proof, we assume that $G=\SL^\epsilon_n(p^a)$ with $\epsilon\in\{\pm1\}$.  Here $\epsilon=1$ means $G=\SL_n(p^a)$, $\epsilon=-1$ means $G=\SU_n(p^a)$, and $|Z(G)|=\gcd(n,p^a-\epsilon)$.  Recall from above that we may assume $n$ is not a power of $q$.  Writing $\wt{G}=\GL_n^\epsilon(p^a)$, we further have \[G^\ast\cong \wt{G}/Z(\wt{G})\hbox{ and } (G^\ast)'\cong GZ(\wt G)/Z(\wt G)\cong S,\] and we will make these identifications.
  
Let $\wt{Q}$ be a Sylow $q$-subgroup of $\wt{G}$, so that $Q:=\wt{Q}Z(\wt{G})/Z(\wt{G})$ is a Sylow $q$-subgroup of $G^\ast$.   Now, if $Z(Q)\cap (G^\ast)'\neq 1$, we can take $s$ to be a nontrivial element of this intersection, and we are done.  So we may assume $Z(Q)\cap (G^\ast)'=1$, in which case $Z(Q)\cong Z(Q)(G^\ast)'/(G^\ast)'\leq G^\ast/(G^\ast)'$ and we see that $|Z(Q)|$ divides $\gcd(n, p^a-\epsilon)_q$, since $|G^\ast: (G^\ast)'|=\gcd(n, p^a-\epsilon)$, and hence $|Z(\wt{Q})|$ divides $(p^a-\epsilon)_q\cdot\gcd(n, p^a-\epsilon)_q$.

Now, writing $n=a_0+a_1q+a_2q^2+\cdots + a_tq^t$ with $0\leq a_i< q$ for the $q$-adic expansion of $n$, we see by \cite{weir} and \cite{carterfong} that $\wt{Q}$ can be chosen as the direct product
\[\wt{Q}=\prod_{i=0}^t Q_i^{a_i}\] where $Q_i$ is a Sylow $q$-subgroup of $\GL^{\epsilon}_{q^i}(p^a)$, and this identification can be made by embedding the matrices block-diagonally into $\wt{G}$.  Note that $Z(Q_i)$ contains the Sylow $q$-subgroup of $Z(\GL^\epsilon_{q^i}(p^a))\cong C_{p^a-\epsilon}$.  Then $|Z(\wt{Q})|>(p^a-\epsilon)_q^2$, a contradiction, unless $n$ is of the form $q^i+q^j$ for some $0\leq i\leq j$.  (Note that this includes the case $n=2q^i$ for some $i\geq 1$, which can only occur if $q\neq 2$.)  Further, $|Z(\wt{Q})|>(p^a-\epsilon)_q\cdot\gcd(n, p^a-\epsilon)_q$ unless $(p^a-\epsilon)_q\mid n_q$.  Since $n_q=q^i$, this means we  must have $(p^a-\epsilon)_q\mid q^i$.

Now let $\mu$ be any nontrivial element in the Sylow $q$-subgroup of $C_{p^a-\epsilon}$, viewed inside $\mathbb{F}_{p^a}^\times$ or $\mathbb{F}_{p^{2a}}^\times$, corresponding to the cases $\epsilon=1$ and $-1$, respectively.  Consider the element $x=(\mu  I_{q^i}, I_{q^j})\in Z(\GL_{q^i}^\epsilon(p^a))\times Z(\GL_{q^j}^\epsilon(p^a))\leq Z(\wt Q)$.  Then $x$ is a non-central semisimple element of $\wt{G}$ in $Z(\wt Q)$ and satisfies $\det{x}=\mu^{q^i}=1$, where the last equality follows from the fact that $(p^a-\epsilon)_q\mid q^i$.  Hence $x\in G$ and therefore the element $s:=xZ(\wt G)$ is a nontrivial element of $GZ(\wt G)/Z(\wt G)=(G^\ast)'$ in $Z(Q)$, again a contradiction.  
 \end{proof}
 
 Using the construction in the proof of Theorem \ref{thm:mainLietype}, we find an analogue of Theorem \ref{teo: MainEu} and Corollary \ref{complete_symmetric} for the groups $\GL_n^\epsilon(r^a)$.
 
 \begin{thm}\label{thm:complete_typeA}
 Let $q\neq p$ be two prime numbers and write $\pi:=\{p,q\}$.  Let ${G}=\GL_n^\epsilon(r^a)$ for a prime $r$ and $\epsilon\in\{\pm1\}$.   Then $\irrpi{{G}} =\mathrm{Lin}({G})$ if, and only if, there is some $k\geq 0$ such that $(r, n)=(p, q^k)$ and $q\mid(p^a-\epsilon)$, up to reordering $p$ and $q$.
 \end{thm}
\begin{proof}
Note that in this case, we may identify ${G}\cong {G}^\ast$. First assume that $\irrpi{{G}} =\mathrm{Lin}({G})$.  
Since the Steinberg character of $G$ has degree $r^{an(n-1)/2}$, it must be that $r=p$ or $r=q$.  Without loss, we assume that $r=p$.  

 Let $Q$ be a Sylow $q$-subgroup of $G$.  Then any character of degree prime to $q$ can be indexed by a pair $(s, \psi)$, where $s$ centralizes $Q$ and $\psi$ has degree prime to $q$, using \eqref{DL}.  However, $\mathrm{Lin}(G)$ is comprised of characters indexed by pairs $(z, 1_G)$ for $z\in Z(G)$ (see \cite[Proposition 13.30]{dignemichel}), so it follows that there are no non-central semisimple elements of $G$ centralizing $Q$.   In particular, this shows that $q\mid (p^a-\epsilon)$, since otherwise $q\nmid |Z({G})|$, so $Q\cap Z(G)=1$, yielding non-central semisimple elements of $G$ contained in $Z(Q)$.  Further, arguing as in the proof of Proposition \ref{thm:mainLietype}, we see $n=q^k$ for some integer $k$, since otherwise $|Z(Q)|\geq (p^a-\epsilon)_q^2$ and there is some semisimple element in $Z(Q)$ not contained in $Z(G)$.
 
 Conversely, assume that $r=p$, $n=q^k$, and $q\mid (p^a-\epsilon)$. Since the centralizer of a semisimple element $s\in {G}$ is a product of groups of the form $\GL_{n_i}^{\epsilon_i}(p^{a_i})$, \cite[Theorem 6.8]{mallehz} and \cite[Proposition 13.20]{dignemichel} yield that the only unipotent characters of $C_{{G}}(s)$ with degree prime to $p$ are linear.  It therefore suffices to show that if $s\in G$ is semisimple and $Q\leq C_G(s)$, then $s\in Z(G)$.
 
 Let $s\in G$ be a semisimple element centralizing $Q$.  For $i\geq 0$, let $Q_i$ and $T_i$ denote a Sylow $q$-subgroup of $\GL^\epsilon_{q^i}(p^a)$ and $\sf S_{q^i}$, respectively.    Using \cite{carterfong} and \cite{weir}, we have $Q=Q_k$ is of the form $Q_1\wr T_{k-1}$.  
 
Suppose $q$ is odd.  If $\epsilon=1$, then further $Q=Q_0\wr T_{k}\cong C_{(p^a-1)_q} \wr T_{k}$.  In particular, in this case, the normal subgroup $C_{(p^a-1)_q}^{q^k}$ of $Q$ may be viewed as the Sylow $q$-subgroup of a maximally split torus, consisting of diagonal matrices.  Since $s$ must centralize this subgroup, we have $s$ is also a diagonal matrix.   Further, with this identification, $T_k$ acts via permuting the copies of $C_{(p^a-1)_q}$, and hence the diagonal entries of a diagonal matrix. Then since $s$ must commute with this subgroup, we see that $s$ is a scalar matrix, so is in $Z(G)$. 

 Now let $\epsilon=-1$, so $G=\GU_{q^k}(p^a)\leq \GL_{q^k}(p^{2a})$.  Then since $q\mid (p^a+1)$, \cite[4(ii)]{weir} yields that $Q$ is a Sylow $q$-subgroup of $\GL_{q^k}(p^{2a})$.  As $q\mid (p^{2a}-1)$, the previous paragraph shows that $s$ must be an element of $Z(G)$.  
 
 Finally, assume $q=2$ and $\epsilon\in\{\pm1\}$.  Then by \cite{carterfong}, we have $N_G(Q)\cong Q\times C_{(p^a-\epsilon)_{2'}}$, where the factor $C_{(p^a-\epsilon)_{2'}}$ is embedded naturally as the largest odd-order subgroup of $Z(G)$, which shows that an element of $C_G(Q)$ must be a member of $Z(Q)Z(G)$.  But since $x\in Z(Q)$ must commute with the action of $T_{k-1}$,  $Z(Q)$ must be comprised of elements of $Q_1^{2^{k-1}}$ whose components are all the same.  Further, these components must be in $Z(Q_1)$.  Considering the description in \cite{carterfong} of the Sylow 2-subgroups $Q_1$ of $\GL^\epsilon_2(p^a)$, we see that $Z(Q)$ therefore consists of scalar matrices, and hence elements of $Z(G)$.
\end{proof}

\begin{cor}\label{cor:complete_typeA}
 Let $q\neq p$ be two prime numbers and write $\pi:=\{p,q\}$.  Let ${G}=\GL_n^\epsilon(r^a)$ for a prime $r$ and $\epsilon\in\{\pm1\}$.   Then $\Gamma'(G)$ is not complete if, and only if, there is some $k\geq 0$ such that $(r, n)=(p, q^k)$ and $q\mid(p^a-\epsilon)$, up to reordering $p$ and $q$.  In this case, the subgraphs of $\Gamma'(G)$ defined by $\pi(|G|)\smallsetminus\{p\}$ and by $\pi(|G|)\smallsetminus\{q\}$ are complete, and there is no edge between $p$ and $q$.
\end{cor}

%\subsection{Extendible $\pi'$ Characters} 
To conclude this section, we make some remarks about the corresponding statement to Proposition \ref{prop:Eunew} for groups of Lie type. We first remark that if $S$ is a simple group of Lie type defined over a field of characteristic $r$ and $\pi$ is a set of primes not containing $r$, then there exists a member of $\irrpi S$ that extends to $\aut S$, taking for example the Steinberg character.  Hence we are interested in the question of when there exists a member of $\irrpi S$ that extends to $\aut S$, where $\pi=\{p, q\}$ with $p=r$.  

For example,  for $S={\rm PSp}_{2n}(p^a)$ with $p$ odd and $\pi=\{p,2\}$, \cite[Theorem 6.8]{mallehz} and \cite[Proposition 4.8]{malleext} yield that there are no such characters.  More generally, given \cite[Theorem 6.8]{mallehz}, an important step is to determine when there exists a semisimple character of degree prime to $q$ that extends to $\aut S$.   
Although the general consideration of this question is beyond the scope of the current article, we make use of some of the techniques used already in this section to answer it in the case of $\PSL_n(p^a)$ with $\pi=\{p,2\}$.

 %While this is an interesting question, and one relevant to determining the normal structure of groups satisfying $\irrpi G=\mathrm{Lin}(G)$ (see Lemma \ref{lem:silly} below), completely classifying simple groups of Lie type and pairs $\pi$ such that there is a member of $\irrpi S$ that extends to $\aut S$ is beyond the scope of the present article, and will be the topic of future work.  Hence here we make only a few initial observations that make use of some of the techniques used already in this section.  

\begin{prop}\label{prop:SLnew}
Let $S=\PSL_n^\epsilon(p^a)$ be simple for an odd prime $p$ and $\epsilon\in\{\pm1\}$.  Let $q$ be a prime dividing $p^a-\epsilon$ and let $\pi:=\{p,q\}$.  If $n$ is not of the form $n=q^i$ or $n=q^i+q^j$ for any $0\leq i\leq j$ and $n\neq 4q^i$ if $4\nmid (p^a-\epsilon)$, then there exists a non-principal $\theta\in\irrpi S$ that extends to $\aut S$.  Further, if $\pi=\{p,2\}$, then there exists a non-principal $\theta\in\irrpi S$ that extends to $\aut S$ if, and only if, $n\neq 2^i$ for any $i\geq 1$.
\end{prop}
\begin{proof}

Assume $n\neq q^i$ for any $i\geq 1$, and further assume $n\neq q^i+q^j$ for any $0\leq i\leq j$ if $q>2$ and that $n\neq 4q^i$ if $p^a\equiv -\epsilon\pmod 4$.  Then we may write $n=a_0+a_1q+a_2q^2+.... +a_tq^t$ with $0\leq a_i<q$ for each $i$ and $\sum_{i=1}^ta_i>2$ if $q$ is odd and $\sum_{i=1}^ta_i>1$ if $q=2$.  Let $G=\SL^\epsilon_n(p^a)$ and $\wt{G}=\GL^\epsilon_n(p^a)$, so $\wt{G}^\ast\cong \wt{G}$, and we make this identification.  Let $\wt{Q}$ be a Sylow $q$-subgroup of $\wt{G}$, so that we may write $\wt{Q}=\prod_{i=1}^t Q_i^{a_i}$ embedded block-diagonally, as before, where $Q_i$ is a Sylow $q$-subgroup of $\GL^\epsilon_{q^i}(p^a)$. 

Let $i$ be the first index for which $a_i\neq 0$.  If $q=2$, let $s\in \wt{G}$ be of the form $\mathrm{diag}(-I_{2^i},I_{n-2^i})$.  If $q>2$, $p^a\equiv\epsilon \pmod 4$, and $n=4q^i$, let $s=\mathrm{diag}(\omega I_{q^i}, \omega^{-1} I_{q^i}, I_{2q^i})$ where $\omega\in C_{p^a-\epsilon}$ has order $4$.  Otherwise, let $s$ be of the form $\mathrm{diag}(-I_{2q^i}, I_{n-2q^i})$ if $a_i\geq 2$.  If $a_i=1$, let $j$ be the next index for which $a_j\neq 0$, and let $s$ be of the form $\mathrm{diag}(-I_{q^i}, -I_{q^j}, I_{n-q^i-q^j})$.  Then note that $\det s=1$, so $s\in G=\wt{G}'$ and the semisimple character $\chi_s$ corresponding to $(s, 1_{C_{\wt{G}}(s)})$ is trivial on $Z(\wt{G})$.  Further, $s\in \prod_{i=1}^t Z(\GL_{q^i}(p^a))^{a_i}$, so it centralizes $\wt{Q}$, and hence $\chi_s \in\irrpi{\wt{G}}$.  Further, since conjugacy classes of semisimple elements in $\wt{G}$ are determined by the eigenvalues, we see that $sz$ is not conjugate to $s$ for any nontrivial $z\in Z(\wt{G})$. Then since the number of irreducible constituents of the restriction of $\chi_s$ to $G$ is exactly the number of irreducible characters $\theta \in \Irr(\widetilde{G}/G)= \{{\chi}_z \mid z \in Z(\widetilde{G}^{\ast})\}$ satisfying ${\chi}_s\theta = {\chi}_s$,  \cite[13.30]{dignemichel} yields that $\chi_s|_G$ is irreducible, yielding a character $\chi$ of $\irrpi S$.  Further, $s$ is $\out S$-invariant by construction, and hence so is $\chi_s$ .   We may then apply \cite[Proposition 3.4 and proof of Lemma 2.13]{spath12} to see that $\chi$ extends to $\aut S$.  The last statement follows by considering Lemma \ref{lem:silly} below together with Theorem \ref{thm:complete_typeA}.
\end{proof}

\section{Groups whose $\pi'$-degree characters are linear}\label{Further}
In this section we provide further discussion on finite groups $G$ satisfying $\mathrm{Irr}_{\pi'}(G)=\mathrm{Lin}(G)$, where $\pi=\{p,q\}$ for some distinct primes $p$ and $q$.

\smallskip

The Navarro-Wolf theorem (see Theorem \ref{NW}) can be slightly generalized by using the following consequence of Wolf's $\pi$-version of the McKay conjecture for $\pi$-separable groups.

\begin{lem}\label{piMcKay} Let $M \nor G$ and let $\pi$ be a set of primes. Suppose that $G/M$ is $\pi$-separable and let $H/M$ be a Hall $\pi$-subgroup of $G/M$. Write $N=\norm G H$. Then $|\irrpi G|=|\irrpi N|$.
\end{lem}
\begin{proof} Given $\chi \in \irrpi G$. Note that $\chi$ lies over a single $N$-orbit of $H$-invariant $\irrpi M$. Let $\phi$ be under $\chi$. Then $\phi \in \irrpi M$ and $|G:G_\phi|$ is a $\pi'$-number. Hence $H^{x^{-1}}/M\sbs G_\phi/M$ for some $x \in G$. In particular, $H\sbs G_{\phi^x}$. Write $\theta=\phi^x$, then $\theta$ is $H$-invariant and lies under $\chi$. Moreover, if $\theta'\in \irr M$ is another $H$-invariant character lying under $\chi$, then $\theta'=\theta^y$ for some $y \in G$. Then $H, H^y\sbs G_{\theta^y}$. Hence, there exists some $x \in G_{\theta^y}$ such that $H=H^{yx}$ and $\theta^{yx}=\theta^y=\theta'$, as claimed. 
The same happens for each $\psi \in \irrpi N$ by the same argument.

 Let $\Theta$ be a set of representatives of the $N$-orbits on the set of $H$-invariant $\irrpi M$. Then 
$$\irrpi G=\dot{\bigcup}_{\theta \in \Theta}\irrpi{G|\theta} \text{ \ and \ } \irrpi N=\dot{\bigcup}_{\theta \in \Theta}\irrpi{N|\theta}.$$
It will be enough to show that $|\irrpi{G|\theta}|=|\irrpi{N|\theta}|$ for each $\theta \in \Theta$. Since $|G:G_\theta|$ and $|N:N_\theta|$ are $\pi'$-numbers, by the Clifford correspondence we may assume $G_\theta=G$. By using projective representations we can find a character triple $(G^*, M^*, \theta^*)$ isomorphic to $(G, M, \theta)$ such that $M^*\sbs \zent{G^*}$. In particular, $G^*$ is $\pi$-separable. We can now apply Corollary 1.15 \cite{W} to get that $|\irrpi {G^*|\theta^*}|=|\irrpi{N^*|\theta^*}|$, with some caution. In Corollary 1.15 of \cite{W}, we should take $\pi=\pi(|G^*|)$ so that $B_\pi(X)=\irr{ X}$ for every $X\leq G^*$, and then let $\omega$ be equal to the set of primes $\pi$ given by our statement. 
\end{proof}

Let $G$ be a finite group. Then the solvable residual of $G$ is the smallest normal subgroup $M$ of $G$ such that $G/M$ is solvable. In particular, $M$ is perfect. Notice that for every $M\sbs H\leq G$, we have that $M\sbs H'$.

\begin{thm}\label{solvable_residual} Let $G$ be a group and let $\pi$ be a set of primes. Write $M$ for the solvable residual of $G$, and let $H/M$ be a Hall $\pi$-subgroup of $G/M$. 
Then $\irrpi G={\rm Lin}(G)$  if, and only if, $\norm {G/M}{H/M}\cap G'/M=H'/M$ and $H$ acts on $\irrpi M$ with fixed points $\{ 1_M \}$.
\end{thm}
\begin{proof}
To prove the ``only if'' implication, note that the first hypothesis implies that $\irrpi {G/M}={\rm Lin}(G/M)$ by \cite{NW}. Let $\chi \in \irrpi G$ and let $\theta \in \irr M$ be under $\chi$. If $\theta=1_M$, then $\chi \in \irrpi {G/M}={\rm Lin}(G/M)={\rm Lin} (G)$. Otherwise $\theta\neq 1_M$ is not $H$-fixed by hypothesis. Hence $|G:G_\theta|$ is divisible by some prime in $\pi$, and $\chi(1)$ is not a $\pi'$-number, a contradiction.

To prove the converse, first notice that $\irrpi G={\rm Lin} (G)$ implies that $\irrpi{G/M}={\rm Lin}(G/M)$. Write $N=\norm G H$. By \cite{NW} we have that $N/M\cap G'/M=H'/M$. Then $G'H/M\nor G/M$ and, by the Frattini argument, we have that $G/M=(G'H/M)N/M$. Hence $G=G'N$ and
$$|\irrpi G|=|\irr{G/G'}|=|\irr{N/H'}|\leq |\irrpi N| \, ,$$ 
as every $\irr{N/H'}$ is linear. By Lemma \ref{piMcKay}, the equality $|\irrpi G|=|\irrpi N|$ forces $\irrpi N={\rm Lin} (N)$.
Assume that $1_M\neq \theta\in \irrpi M$ is $H$-fixed. Since $M$ is perfect, then $\theta(1)>1$ and $o(\theta)=1$. Then $\theta$ extends to some $\phi \in \irr H$ by Corollary 6.28 of \cite{Isa76}. Let $\psi \in \irr{N|\phi}$, then $\psi \in \irrpi N$ (by  Corollary 11.29 of \cite{Isa76}) is non-linear, a contradiction.  
\end{proof}

Note that if $\pi$ consists of two primes, then $\irrpi M \supsetneq \{ 1_M\}$ by Theorem A and hence $H/M \in {\rm Hall}_\pi(G/M)$ above is non-trivial.
\enskip

By using the following Lemma we can assure that certain non-abelian simple groups do not appear as composition factors of groups satisfying $\irrpi G={\rm Lin}(G)$. (Note that non-abelian composition factors of $G$ appear as composition factors of its solvable residual).

\begin{lem}\label{lem:silly} Let $S$ be a non-abelian simple composition factor of a group $G$. Let $\pi$ be a set of primes. If some non-principal $\theta \in \irrpi S$ extends to $\aut S$, then ${\rm Lin}(G) \subsetneq \irrpi G$.
\end{lem}
\begin{proof}
Since $\irrpi {G/M}={\rm Lin}(G/M)$ whenever $M\nor G$, by conveniently taking quotients we may assume that $G$ has a minimal normal subgroup $N$ which is the product of the $G$-conjugates of $S$. The statement then follows from Lemma 5 of \cite{BCLP}.
\end{proof}

%The above lemma used together with some of the observations made in previous sections of this article already allow us to exclude various simple groups as composition factors of a group $G$ satisfying $\irrpi G ={\rm Lin}(G)$. These observations are made more precise in the following final remark. 

\begin{rem}
Let $\pi=\{p,q\}$, let $G$ be such that $\irrpi G ={\rm Lin}(G)$ and let $S$ be a simple composition factor of $G$. 
Then using Proposition \ref{prop:Eunew}, we deduce that $S$ can not be an alternating group ${\sf A}_n$, unless $n=2q\6k+1=p\6m$ or $n=2q\6k=p\6m+1$ for some $k,m\in\mathbb{N}$. 

The observations in Proposition \ref{prop:SLnew} and the discussion before it further yield that if $S$ is a simple group of Lie type, then the defining characteristic is in $\pi$, and if $\pi=\{p,2\}$,  then $S$ is not $\PSL_n(p^a)$ unless $n=2^k$ for some $k\in\N$.    

%Finally, we consider sporadic simple groups. 
Using Lemma \ref{lem:silly} together with GAP and the Character Table Library, we can see that if $\pi=\{2,p\}$ then the only simple sporadic groups that can possibly appear as composition factors of $G$ are $\mathrm{J}_3$ for $p=5$, $\mathrm{McL}$
for $p=7$, $\mathrm{Suz}$ for $p=13$ and 
$\mathrm{He}$ for $p=17$.
\end{rem}

\end{document}